\documentclass[12pt, reqno]{amsart}
\usepackage{amsmath,amssymb,amsthm}

 2

\usepackage[colorlinks=true,linkcolor=blue,citecolor=blue,pdfpagelabels=false]{hyperref}

\newtheorem{thm}{Theorem}[section]

\newtheorem{prop}[thm]{Proposition}
\newtheorem{cor}[thm]{Corollary}

\newcommand{\thmref}[1]{Theorem~\ref{#1}}

\newcommand{\rmkref}[1]{Remark~\ref{#1}}
\newcommand{\propref}[1]{Proposition~\ref{#1}}
\newcommand{\corref}[1]{Corollary~\ref{#1}}

\theoremstyle{remark}
\newtheorem{rmk}{Remark}[section]

\newenvironment{acknowledgements}{\bigskip\textbf{Acknowledgements.}}{}

\newcommand{\nequiv}{\not\equiv}
\renewcommand{\geq}{\geqslant}
\renewcommand{\leq}{\leqslant}

\begin{document}

\title
{Non-vanishing of  Hilbert Poincar\'e series}
\author[M. Kumari]{Moni Kumari}
\address{National Institute of Science Education and Research, HBNI, Bhubaneswar, Via-Jatni, Khurda, Odisha, 752050, India.}

\email{moni.kumari@niser.ac.in}

\date{\today}

\subjclass[2010]{Primary 11F41; Secondary 11F30.}

\keywords{Hilbert modular forms, Poincar\'e series, non-vanishing}
\begin{abstract}
We prove some non-vanishing results of Hilbert Poincar\'e series.  
We derive these results, by showing that the Fourier coefficients of Hilbert Poincar\'e series satisfy some nice orthogonality relations for sufficiently large weight as well as for sufficiently large level.
To prove later results, we generalize a method of E. Kowalski et. al.
\end{abstract}
\maketitle

\section{Introduction}
The vanishing or non-vanishing of Poincar\'e series of integral weight for the group $SL_{2}(\mathbb{Z})$ is a mysterious problem. Here, there is a conjecture that none of the Poincar\'e series  vanish. 
Several papers have appeared to investigate this conjecture. The first non-trivial result towards this question was given by R. A. Rankin \cite{rankin} in 1980, who showed that 
there are constants $B>4\log{2}$ and $k_0$ such that the Poincar\'e series $\mathcal{P}_{n,k,1}(z)$ does not vanish identically for 
\begin{equation}
n\leq k^2 \exp \bigg(\frac{-B\log k}{\log \log k}\bigg),
\end{equation}
provided $k\geq k_0,$ where $\mathcal{P}_{n,k,1}(z)$  is the $n$-th Poincar\'e series of integral weight $k$ and of level $1$. 
His proof uses the fact that the Fourier coefficients of Poincar\'e series has an explicit formula as an infinite series involving Bessel functions and Kloosterman sums and sharp estimates for the magnitude of Kloosterman sums. Later, J. Lehner \cite{lehner} and C. J. Mozzochi \cite{mozz} generalized the Rankin's result for an arbitrary Fuchsian group and for the congruence subgroup $\Gamma_{0}(N)$ respectively. In 1986 E. Gaigalas \cite{gai}, using Weil's estimate for the Kloosterman sum proved the following. ``For any $m\in \mathbb{N}$ there exist infinitely many $k\in 2\mathbb{N}$ for which the $m$-th Poincar\'e series of weight $k$ (with respect to any finite index subgroup of $SL_2(\mathbb{Z})$) is not identically zero." It is notable that the non-vanishing of Poincar\'e series is related to the famous conjecture of Lehmer \cite{lehmer}, which says that $\tau(n)\neq 0$, for all $n\geq 1,$ where $\tau$ is the Ramanujan $\tau$-function.

As we know that the space of Hilbert modular forms is a generalization of the space of elliptic modular forms, it is natural to ask the vanishing or non-vanishing of Hilbert Poincar\'e series. To the best of our knowledge, in the literature there is no result concerning this problem. In this article, we give an answer for this question in terms of weight as well as level aspects. More precisely, we prove that for a fixed level $\mathcal{I}$ and for any finite set $\mathcal{A}\subset \mathcal{O}_F^*$, there exist a positive constant $k_{\mathcal{A}}$ such that for all $k\geq k_{\mathcal{A}}$ and for any $\nu \in\mathcal{A}$
\begin{equation}
\mathcal{P}_{\vec{k},\nu,\mathcal{I}}(z)\nequiv 0.
\end{equation}
See section \ref{notation} for the notations.
We also prove an analogous result with respect to level. Note that in a particular case we get a generalization of Gaigalas's result for Hilbert Poincar\'e series. 

For the proof of our main results, we first prove (see \thmref{weight} and \thmref{level}) that the Fourier coefficients of Hilbert Poincar\'e series satisfy some nice orthogonality relations with respect to weight as well as level. These intermediate results generalize the work of E. Kowalski et. al. \cite{saha}. The main ingredients to prove these orthogonality relations are basic Fourier analysis for functions of several variables and dominated convergence theorem. These results may be of independent interest.

The paper is organized as follows. Section \ref{notation} contains some notations and a brief introduction about Hilbert modular forms. Here we also state the results of Kowalski et. al. \cite{saha} for Poincar\'e series. Furthermore, their method has been extended to Hilbert Poincar\'e series in section \ref{intermediate}. The main results of the paper which is about the non-vanishing of Hilbert Poincar\'e series are given in section \ref{mainresult}.
\section{Notations and preliminaries}\label{notation}
Let $F$ be a totally real number field of degree $n$ over $\mathbb{Q}$ and $\mathcal{O}_{F}$ be its ring of algebraic integers. Assume that $\sigma_1, \sigma_2,...,\sigma_n$ denote the real embeddings of $F$. We write $\alpha_i= \sigma _i(\alpha)$ for $\alpha \in F ~\mbox{and}~ 1\leq i\leq n.$ The trace and norm of $\alpha \in F$ are defined by ${\rm tr}(\alpha)=\sum_{i=1}^n \alpha_i$ and $N(\alpha)=\prod_{i=1}^n \alpha_i$ respectively. For $\alpha \in \mathbb{C}^n$, the trace and norm are defined by the sum and the product of its components respectively. More generally, if $ c= (c_1,c_2,\cdots, c_n), d=(d_1,d_2,\cdots, d_n), z=(z_1,z_2,\cdots,z_n)~ \mbox{and}~m=(m_1,m_2,\cdots,m_n)\in \mathbb{C}^n,$ then the norm and trace are define by
  \begin{equation*}
 N(cz+d):= \prod_{i=1}^{n}(c_iz_i+d_i) ~~~~~~~~{\rm ~and~~~}~~~~~~~~~{\rm tr}(mz):= \sum _{i=1}^{n}m_iz_i.
\end{equation*}
For $\alpha \in \mathcal{O}_{F},$ we write $\alpha \succeq 0$ to demonstrate that either $\alpha=0 $ or $\alpha$ is totally positive (means all the conjugates of $\alpha$ are positive) and $\alpha\gg 0$ for $\alpha$ to be totally positive.

Let $GL^+_{2}(\mathbb{R})$ be the set of all matrices in $GL_{2}(\mathbb{R})$ with positive determinant. We know that the group $GL^+_{2}(\mathbb{R})$ acts on the upper half plane 
$\mathbb{H}=\{x+iy\in \mathbb C:y>0\}$ via,
$$\left(\begin{array}{cc}a&b\\c &d\end{array}\right)z=\frac{az+b}{cz+d}.$$
Now by using the above action, we define an action of $GL_{2}^+(\mathcal{O}_{F})$) on $\mathbb{H}^n$. Using all the embeddings of $F$ and fixing their order, we embed $GL_{2}(F)$ into $GL_{2}(\mathbb{R})^n.$ The image of $GL_{2}(\mathcal{O}_{F})$ in $GL_{2}(\mathbb{R})^n$ is discrete. We denote the set of elements in $GL_{2}(F)$ (respectively  $GL_{2}(\mathcal{O}_{F})$) with totally positive determinant by $GL^+_{2}(F)$ (respectively  $GL_{2}^+(\mathcal{O}_{F})$). For $g=(g_1,g_2,\cdots,g_n)\in GL_{2}^+(\mathbb{R})^n$ and $z=(z_1,z_2,\cdots,z_n)\in \mathbb{H}^n$, we define
$$ gz:=(g_1z_1,g_2z_2,\cdots,g_nz_n),$$
which deduces an action of $GL^+_{2}(F)$ and hence of $GL_{2}^+(\mathcal{O}_{F})$ on $\mathbb{H}^n.$

In this paper, we shall work with the Hilbert modular group
$$ \Gamma_{F}=SL_{2}(\mathcal{O}_{F}):= \bigg\{\left(\begin{array}{cc}a&b\\c &d\end{array}\right): a,b,c,d\in \mathcal{O} _{F},~ ad-bc=1 \bigg\}, $$
and its congruence subgroups of level $\mathcal{I}$, which is defined by
$$\Gamma_{0}(\mathcal{I}):= \bigg\{\left (\begin{array}{cc}\alpha&\beta\\ \gamma & \delta\end{array}\right) \in {\Gamma_{F}: \gamma \in \mathcal{I}}\bigg\},$$
where $\mathcal{I} \subseteq \mathcal{O} _{F}$ is a non-zero integral ideal.

Let $g=(g_1,g_2,\cdots,g_n)\in GL_{2}^+(\mathbb{R})^n$, $z=(z_1,z_2,\cdots,z_n)\in \mathbb{H}^n$, and $k=(k_1,k_2,\cdots,k_n)\in \mathbb{Z}^n$. We define
\begin{equation*}
\mu(g,z)^k:=\prod_{j=1}^{n}({\rm det}g_j)^{-k_j/2}(c_jz_j+d_j)^{k_j},
\end{equation*}
where $g_j=\left(\begin{array}{cc}*&*\\c_j &d_j\end{array}\right).$ Furthermore, for a function $f$ defined on $\mathbb{H}^n$, we set
$$(f|{_k}g)(z)=\mu(g,z)^{-k}f(gz).$$ 

\subsection{Hilbert modular forms}
A Hilbert modular form of weight $k \in \mathbb{N}^{n}_{0}$ for a congruence subgroup $ \Gamma$ of $\Gamma_{F}$ is a holomorphic function $f:\mathbb{H}^n \rightarrow \mathbb{C}$ such that 
$$ f\mid_k \gamma= f,~\mbox{for~ all}~\gamma \in \Gamma.$$
For $n=1$, we also need holomorphicity condition at the cusps of $\Gamma$. Note that for $n>1$, a Hilbert modular form is automatically holomorphic at the cusps by 
the Koecher principle \cite[Sec 1.4]{gar}. In addition, $f$ is called a Hilbert cusp form if it vanishes at all the cusps of $\Gamma.$ Let $M_{k}(\Gamma)$ denotes the space of Hilbert modular forms of  weight $k \in \mathbb{N}_{0}^{n}$ for the congruence subgroup $ \Gamma$ and $S_{k}(\Gamma)$ be the subspace of cusp forms. These are finite dimensional complex vector spaces and $S_{k}(\Gamma)$ is a Hilbert space with respect to the Petersson inner product
\begin{equation*}\label{inner}
\langle f,g\rangle_{} :=\int_{\Gamma\setminus \mathbb{H}^n} f(z)\overline{g(z)}y^k \frac{dxdy}{y^2},
\end{equation*} 
where $z=x+iy, ~dx=dx_{1}\cdots dx_{n}~\mbox{and}~dy=dy_{1} \cdots dy_{n}.$
If $f\in M_{k}(\Gamma_{0}(\mathcal{I}))$, where $\mathcal{I} \subseteq \mathcal{O} _{F}$ is a non-zero integral ideal, then we call $f$ to be a Hilbert modular form of weight $k$ and of level $\mathcal{I}$. Note that if $F=\mathbb{Q}$, then $M_{k}(\Gamma)$ is the space of elliptic modular forms.

By the Koecher principle, $f\in M_{k}(\Gamma)$ has a Fourier expansion at the cusp $\infty$ of the form
\begin{equation*}
f(z)= \sum \limits_{\substack{m\in\Lambda_{\Gamma}^*\\m\succeq 0}} a_{m}e^{2\pi i {\rm tr}(m z)}, 
\end{equation*}
where 
\begin{equation}\label{lambdaf}
\Lambda_{\Gamma}=\bigg\{\mu \in F:\left(\begin{array}{cc}1&\mu\\0 &1\end{array}\right)\in\Gamma \bigg\},~~
\Lambda_{\Gamma}^*=\big\{\mu\in F:{\rm tr}(\mu \Lambda_{\Gamma})\subseteq \mathbb{Z}\big\}.
\end{equation}
Here $\Lambda_{\Gamma}^*$ is called the dual space of $\Lambda_{\Gamma}$. Note that if $\Gamma=\Gamma_{0}(\mathcal{I})$ then $\Lambda_{\Gamma}=\mathcal{O}_{F}$ and $\Lambda_{\Gamma}^*=\mathcal{O}_{F}^*$.

Now we introduce some notations which will be used in later sections. For an integer $x \in\mathbb{N}_{0},$
 we denote $\vec{x}= (x,x, \cdots, x)\in \mathbb{N}_{0}^{n}.$ For $\nu= (\nu_{1},\nu_2,\cdots,\nu_{n}) \in \mathbb{N}_{0}^{n}~ \mbox{and}~ z=(z_1, z_2,\cdots, z_n)\in \mathbb{C}^n,$ we put
 $$ |\nu|= \sum_{i=1}^{n} \nu_{i},~\nu!= \prod_{i=1}^{n}\nu_{i}!~\mbox{and}~z^{\nu}=\prod_{i=1}^{n} z_{i}^{\nu_i}.$$
For an integral ideal $\mathcal{I}$ of $\mathcal{O}_F$, we denote the norm of  $\mathcal{I}$ by $\mathcal{N}(\mathcal{I}):=[\mathcal{O}_F:\mathcal{I}].$ For $z=(z_1, z_2,\cdots, z_n)\in \mathbb{H}^n,$ we always write ${\rm{Im}}(z)  =({\rm{Im}}(z_1), {\rm{Im}}(z_2),\cdots, {\rm{Im}}(z_n))$. Throughout the paper, we write Poincar\'e series for Poincar\'e series of the group $SL_2(\mathbb Z)$ and its congruence subgroups.
\subsection{Hilbert Poincar\'e Series:}
Let $\mathcal{I}\subseteq\mathcal{O}_{F}$ be a non-zero integral ideal and $\Gamma_{0}(\mathcal{I})$ be the associated congruence subgroup. For a totally positive element $\nu$ of $\mathcal{O}_{F}^*$ and weight $k=(k_1,k_2,\cdots,k_n)$ $(k_j>2, ~j=1,2,\cdots,n)$, we define the $\nu$-th Hilbert Poincar\'e series as follows:
\begin{equation}\label{pdefinition}
\mathcal{P}_{k,\nu,\mathcal{I}}(z)=\sum_{M\in\Gamma_{\infty}\setminus \Gamma_{0}(\mathcal{I})}\mu(M,z)^{-k}e^{2\pi i{\rm tr}(\nu(Mz))},
\end{equation}
where $\Gamma_{\infty}=\bigg\{\left(\begin{array}{cc}1& \mu \\0 &1\end{array}\right): \mu\in\mathcal{O} _{F}\bigg\}.$ It is well known that $\mathcal{P}_{k,\nu,\mathcal{I}} \in S_{k}(\Gamma_{0}(\mathcal{I}))$. 

For more details on the theory of Hilbert modular forms, we refer \cite{fre} and \cite{gar}.
\subsection{Orthogonality of Fourier coefficients of Poincar\'e series} Here we recall some results obtained by Kowalski et. al. \cite{saha}, concerning the orthogonality properties of the Fourier coefficients of Poincar\'e series. In the next section we generalize them for Hilbert Poincar\'e series. 

Let $k > 2$ and $m\geq 1$ be integers. Let $\mathcal{P}_{m,k,q}(z)$ be the $m$-th Poincar\'e series of weight $k$ for the group $\Gamma_{0}(q)$. Recently, Kowalski et. al. \cite{saha} proved that the Fourier coefficients $p_{m,k,q}(n)$ of the Poincar\'e series $\mathcal{P}_{m,k,q}(z)$ satisfy the following orthogonal relations with respect to weight $k$ as well as the level $q$.
\begin{prop}\label{k}
With notations as above, for fixed $m\geq 1$ and $n\geq 1$, we have
\begin{equation*}
\lim_{k\rightarrow\infty}p_{m,k,1}(n)=\delta(m,n),
\end{equation*}
where $\delta(\cdot, \cdot)$ is the Kronecker symbol.
\end{prop}
Furthermore, they also showed similar result with respect to another important parameter, the level $q$.

\begin{prop}\label{q}
With notations as above, for fixed $k\geq 4, m$ and $n$, we have
\begin{equation*}
\lim_{q\rightarrow \infty}p_{m,k,q}(n)=\delta(m,n).
\end{equation*}
\end{prop}
Note that the above orthogonality relations can be obtained by applying the trace formula to Poincar\'e series. But in \cite{saha}, the authors gave a completely different and a soft proof. Furthermore, in the same paper they have also obtained a similar orthogonality results for the coefficients of Siegel Poincar\'e series

\section{Intermediate results: Orthogonality relations}\label{intermediate}
In this section, we extend \propref{k} and \propref{q} for Hilbert Poincar\'e series, which will be used to get our main results. We remark that these results may be of independent interest.

Let $p_{k,\nu,\mathcal{I}}(\mu)$ be the $\mu$-th Fourier coefficient of the Hilbert Poincar\'e series $\mathcal{P}_{k,\nu,\mathcal{I}}$ defined by \eqref{pdefinition}, where $\mu\in\mathcal{O}_F^*.$ First, we prove that the Fourier coefficients $p_{k,\nu,\mathcal{I}}$ satisfy an orthogonality relation for sufficiently large $k.$ The main ingredient for the proof is dominated convergence theorem. In \thmref{weight}, we shall only work with parallel weight $\vec{k}=(k,k,...,k).$
\begin{thm}{\bf [Weight aspect]}\label{weight}
Let $\mathcal{P}_{\vec{k},\nu,\mathcal{I}}$ be the $\nu$-${\rm th}$ Hilbert Poincar\'e series of weight $\vec{k}$ and of level $\mathcal{I}$. Then for fixed $\nu\gg 0$ and  $\mu\gg 0$, we have 
$$\lim_{k\rightarrow \infty}p_{\vec{k},\nu,\mathcal{I}}(\mu)= \delta(\nu, \mu).$$
\end{thm}		
	
\begin{proof}
From the theory of Fourier analysis for function of several variables, we have 
\begin{equation*}
 p_{\vec{k},\nu,\mathcal{I}}(\mu)={\rm{vol}}(\Omega)^{-1}\int_{\Omega}\mathcal{P}_{\vec{k},\nu,\mathcal{I}}(z)e^{-2\pi i{\rm tr}(\mu z)}dz,
\end{equation*}
where for an arbitrary $y\in\mathbb{R}_+^n$, $\Omega=\{z=x+iy: x\in\mathcal{O}_F\setminus\mathbb{R}^n\}$. For our purpose, we choose $y$ with $ N(y)>1$. Now taking the limit as ${k\rightarrow \infty}$ on both sides of the above equation, we obtain
\begin{equation}\label{coe}
 \lim_{k\rightarrow \infty}p_{\vec{k},\nu,\mathcal{I}}(\mu)={\rm{vol}}(\Omega)^{-1}\lim_{k\rightarrow \infty}\int_{\Omega}\mathcal{P}_{\vec{k},\nu,\mathcal{I}}(z)e^{-2\pi i{\rm tr}(\mu z)}dz.
\end{equation}
We want to apply the dominated convergence theorem for interchanging the limit and integration on the right hand side of \eqref{coe}. Now we prove the following two assertions, which enable us to apply the dominated convergence theorem  for the sequence $\{\mathcal{P}_{\vec{k},\nu,\mathcal{I}}\}_{k\geq 1}$.
\begin{enumerate}
\item[{\bf (a)}]
For any $z\in\Omega$, as $k\rightarrow\infty$ we have
\begin{equation}\label{convergence}
\mathcal{P}_{\vec{k},\nu,\mathcal{I}}(z)\rightarrow e{^{2 \pi i{\rm tr}(\nu z)}}.
\end{equation}
 \item[{\bf (b)}]
For all $k >2$ and $z\in \Omega$, there exist an integrable function $G$ on $\Omega$ such that 
\begin{equation*}
|\mathcal{P}_{\vec{k},\nu,\mathcal{I}}(z)|\leq G(z).
\end{equation*}
\end{enumerate}

\noindent
To prove the first assertion, we show that as  $k\rightarrow\infty$  exactly one of the term in the series expansion of $\mathcal{P}_{\vec{k},\nu,\mathcal{I}}$ (defined by \eqref{pdefinition}) converges to $e{^{2 \pi i{\rm tr}(\nu z)}}$ and others (individually) tend to $0$. Any term in the series expansion looks like
\begin{equation*}\label{muterm}
\mu(M,z)^{-\vec{k}}e^{2\pi i{\rm tr}(\nu(Mz))},
\end{equation*}
for some $M=\left(\begin{array}{cc}*& * \\ \gamma & \delta\end{array}\right)\in\Gamma_{\infty}\setminus \Gamma_{0}(\mathcal{I})$. Note that $\gamma = 0$ only for one such $M$ and so we can choose $M=I_{\rm id}$ to be a representative. Therefore, in this case
\begin{equation}\label{muequal}
\mu(M,z)^{-\vec{k}}e^{2\pi i{\rm tr}(\nu(Mz))}=e^{2\pi i{\rm tr}(\nu z)},
\end{equation}
for all weight $\vec{k}$ and $z\in \mathbb{H}^n.$ Now suppose $\gamma\neq 0$. For $z=(z_1,z_2,\cdots,z_n)\in \mathbb{H}^n$ it is easy to see that  ${\rm{Im}}({\rm tr}(\nu(Mz)))>0$, which gives
\begin{equation*}\label{mu1}
 |\mu(M,z)^{-\vec{k}}e^{2\pi i{\rm tr}(\nu(Mz))}|\leq |\mu(M,z)^{-\vec{k}}|.
\end{equation*}
Since $0 \neq \gamma \in \mathcal{O}_{F}$, i.e., $N(\gamma)\in \mathbb{Z},$ and hence $N(\gamma)^2\geq 1$. Now
\begin{equation*}\label{mu2}
|\mu(M,z)|^{2}=\prod_{j=1}^{n}|\gamma _jz_j + \delta_j|^2\\
	                                  \geq\prod_{j=1}^{n}(\gamma _j {y}_j)^2\\
                                =N(\gamma)^2N(y)^2\geq N(y)^2,
\end{equation*}
where $y_j={\rm{Im}}(z_j)$ for $1\leq j\leq n.$
Combining the last two inequality, we get
\begin{equation*}\label{muinequality}
|\mu(M,z)^{-\vec{k}}e^{2\pi i{\rm tr}(\nu(Mz))}|\leq N(y)^{-k}.
\end{equation*}
We know that for $z\in\Omega$, $N(y)>1$. Therefore, for any $z\in\Omega$ the above inequality yields 
 \begin{equation}\label{muconvergence}
 |\mu(M,z)^{-\vec{k}}e^{2\pi i{\rm tr}(\nu(Mz))}|\rightarrow 0~~\mbox{as}~~k\rightarrow \infty.
 \end{equation}
 Thus the claim \eqref{convergence} follows from \eqref{muequal} and \eqref{muconvergence}.
 
For the proof of the assertion {\bf (b)}, we define a function $G(z)$ on $\mathbb{H}^n$ by
 \begin{equation*}
G(z)=\sum_{M\in\Gamma_{\infty}\setminus \Gamma_{0}(\mathcal{I})} |\mu(M,z)|^{-3}.
\end{equation*}
It is a well-known (see \cite[lemma 5.7]{fre}) fact that the above series is an absolutely convergent and convergences uniformly in every compact subset of $ \mathbb{H}^n$.
Since the domain $\Omega$ is compact in $\mathbb{H}^n$, therefore $G(z)$ is integrable on $\Omega$, i.e.,
\begin{equation*}
\int_{\Omega}G(z)dz <\infty .
\end{equation*}
Also for any $z\in\Omega$, we obtain
\begin{equation*}
|\mathcal{P}_{\vec{k},\nu,\mathcal{I}}(z)|\leq\sum_{M\in\Gamma_{\infty}\setminus \Gamma_{0}(\mathcal{I})}|\mu(M,z)^{-\vec{k}}e^{2\pi i{\rm tr}(\nu(Mz))}|\leq\sum_{M\in\Gamma_{\infty}\setminus \Gamma_{0}(\mathcal{I})}|\mu(M,z)|^{-3}.
\end{equation*}
In other words, for every positive integer $k>2$ and  $z\in \Omega$, we get 
\begin{equation*}\label{bounded}
|\mathcal{P}_{\vec{k},\nu,\mathcal{I}}(z)|\leq G(z),
\end{equation*}
which proves the required result.

Now the dominated convergence theorem allows us to interchange limit and integration in equation \eqref{coe}.
Therefore using \eqref{convergence}, we have
\begin{equation*}
\lim_{k\rightarrow \infty}p_{\vec{k},\nu,\mathcal{I}}(\mu)={\rm vol}(\Omega)^{-1} \int_{\Omega}e^{2\pi i{\rm tr}(\nu z)}e^{-2\pi i{\rm tr}(\mu z)}dz=\delta(\nu,\mu).
\end{equation*} 
This completes the proof.
\end{proof}	   

Next, we prove that the Fourier coefficients $p_{k,\nu,\mathcal{I}}$ also satisfy a similar orthogonality relation with respect to other important parameter, the level $\mathcal{I}$. The proof, as before, uses the dominated convergence theorem for the sequence $\{\mathcal{P}_{k,\nu,\mathcal{I}}(z)\}_{\mathcal{N}(\mathcal{I})\geq 1}.$ Here we work with any weight $k\in\mathbb{N}_{0}^{n}$, not necessarily parallel weight.

\begin{thm}{\bf [Level aspect]}\label{level}
Let $\mathcal{P}_{k,\nu,\mathcal{I}}$ be the $\nu$-${\rm th}$ Hilbert Poincar\'e series of weight $k\in\mathbb{N}_{0}^{n}$ of level $\mathcal{I}$. Then for fixed $k>2,\nu\gg0$ and $\mu\gg 0,$ we have
    $$\lim_{\mathcal{N}(\mathcal{I})\rightarrow \infty}p_{k,\nu,\mathcal{I}}(\mu)= \delta(\nu, \mu).$$
\end{thm}
\begin{proof}
As before, we start with the integral formula
\begin{equation*}
 p_{k,\nu,\mathcal{I}}(\mu)={\rm vol}(\Omega)^{-1}\int_{\Omega}\mathcal{P}_{k,\nu,\mathcal{I}}(z)e^{-2\pi i{\rm tr}(\mu z)}dz.
\end{equation*}
Taking limit $\mathcal{N}(\mathcal{I})$ tends to infinity on the both sides of the above equation, we get
\begin{equation}\label{coe1}
\lim_{\mathcal{N}(\mathcal{I})\rightarrow \infty} p_{k,\nu,\mathcal{I}}(\mu)={\rm vol}(\Omega)^{-1}\lim_{\mathcal{N}(\mathcal{I})\rightarrow \infty}\int_{\Omega}\mathcal{P}_{k,\nu,\mathcal{I}}(z)e^{-2\pi i{\rm tr}(\mu z)}dz.
\end{equation}
For simplifying the right hand side of the above equation, we need to interchange the limit and integration. For that we use the dominated convergence theorem. To apply this theorem, we prove the following two assertions. First, every element in the sequence $\{\mathcal{P}_{k,\nu,\mathcal{I}}(z)\}_{\mathcal{N}(\mathcal{I})\geq 1}$ is bounded by some integrable function on $\Omega$.  Second, 
 \begin{equation}\label{convergence1}
\mathcal{P}_{k,\nu,\mathcal{I}}(z)\rightarrow e{^{2 \pi i{\rm tr}(\nu z)}},~~~\mbox{as}~~~ \mathcal{N}(\mathcal{I})\rightarrow\infty,\forall~~ z\in\Omega.
\end{equation}

For $k=(k_1,k_2,\cdots,k_n)\in \mathbb{N}^n$ with $k_j>2$ for all $1\leq j\leq n$, we define 
$$G_{k}(z)=\sum_{M\in\Gamma_{\infty}\setminus \Gamma_F}|\mu(M,z)|^{-k}.$$
It is a well-known fact (see \cite[lemma 5.7]{fre}) that the above series is an absolutely convergent and convergences uniformly in every compact subset of $ \mathbb{H}^n$.
Since the domain $\Omega$ is compact in $\mathbb{H}^n$, therefore $G_{k}(z)$ is integrable on $\Omega$, i.e.,
\begin{equation*}
\int_{\Omega}G_{k}(z)dz <\infty .
\end{equation*}
Hence for a fixed $k\in \mathbb{N}^n$ and following the arguments as in the proof of the previous theorem, we have
\begin{equation}\label{bounded1}
|\mathcal{P}_{k,\nu,\mathcal{I}}(z)|\leq G_{k}(z),
\end{equation}
for every integral ideal $\mathcal{I}$ of $\mathcal{O}_{F}$ and for all $z\in \mathbb{H}^n.$ This proves the first assertion.

For the proof of the claim \eqref{convergence1}, we proceed as follows.
Note that $\Gamma_{\infty}\setminus\Gamma_{0}(\mathcal{I})$ is a subset of $\Gamma_{\infty}\setminus\Gamma_{F}$, therefore we write \eqref{pdefinition} as 
\begin{equation}\label{pdefinition1}
\mathcal{P}_{k,\nu,\mathcal{I}}(z)=\sum_{M\in\Gamma_{\infty}\setminus\Gamma_{F}}\Delta_{I}(M) \mu(M,z)^{-k}e^{2\pi i{\rm tr}(\nu(Mz))},
\end{equation}
where
\[\Delta_{I}\bigg(\left(\begin{array}{cc}\alpha&\beta\\ \gamma &\delta\end{array}\right)\bigg)=
   \begin{cases}
    1  & ~\text{if}~  \gamma \in \mathcal{I},\\
    0  &~ \text{otherwise}.\\
   \end{cases} 
     \]
For $M=\left(\begin{array}{cc}*& * \\ \gamma & \delta\end{array}\right)\in\Gamma_{\infty}\setminus \Gamma_{F}$, consider a general term 
\begin{equation}\label{deltaterm}
\Delta_{I}(M) \mu(M,z)^{-k}e^{2\pi i{\rm tr}(\nu(Mz))}
\end{equation}
 in the Poincar\'e series $\mathcal{P}_{k,\nu,\mathcal{I}}(z)$ defined by \eqref{pdefinition1}.
 If $\gamma=0$ then by definition, $$\Delta_{I}(M)=1$$ for all non-zero integral ideal $\mathcal{I}.$ Also $M \in\Gamma_{\infty}$, hence we can take $M$ to be the identity matrix $I_{\rm id}$. Therefore, we have
 \begin{equation}\label{deltaequal}
\Delta_{I}(M) \mu(M,z)^{-k}e^{2\pi i{\rm tr}(\nu(Mz))}=e^{2\pi i{\rm tr}(\nu z)},
\end{equation}
for every non-zero integral ideal $\mathcal{I}.$ Next, if $\gamma\neq 0$, then
\begin{equation*}
\Delta_{I}(M)=0,
\end{equation*}
for all $\mathcal{I}$ with $\mathcal{N}(I)>\mathcal{N}(\gamma\mathcal{O}_F)$.
Hence
\begin{equation}\label{convergence0}
\lim _{\mathcal{N}(I)\rightarrow \infty}\Delta_{I}(M)\mu(M,z)^{-k} e^{2\pi i {\rm tr}(\nu(Mz))}=0.
\end{equation}
The assertion \eqref{convergence1} follows if we take limit $\mathcal{N}(\mathcal{I})$ tends to $\infty$ in the definition
\eqref{pdefinition1} and using \eqref{deltaequal} and \eqref{convergence0}.

Now interchanging the limit and integration in \eqref{coe1} and further using \eqref{convergence1} gives the required result.
\end{proof}

\section{non-vanishing of Hilbert Poincar\'e series}\label{mainresult}
The main goal of the paper is to prove some non-vanishing results for Hilbert Poincar\'e series, which we prove here. Using \thmref{weight} and \thmref{level}, we prove that the Hilbert Poincar\'e series does not vanish identically for sufficiently large weight as well as sufficiently large level respectively.

\begin{thm}\label{nonvanishingw}
Let $\nu\in\mathcal{O}_F^*$, where $\mathcal{O}_F^*$ is the dual space defined by \eqref{lambdaf}. Let $\mathcal{P}_{\vec{k},\nu,\mathcal{I}}$ be the $\nu$-${\rm {th}}$ Hilbert Poincar\'e series of (parallel) weight $\vec{k}$ and level $\mathcal{I}$. Then for fixed $\nu$ and level $\mathcal{I}$, there exist a positive constant $k_0$ such that for all $k>k_0,$ we have
\begin{equation*}
\mathcal{P}_{\vec{k},\nu,\mathcal{I}}\nequiv 0.
\end{equation*}
\end{thm}
\begin{proof}
From \thmref{weight}, we know that for a totally positive $\nu\in\mathcal{O}_F^*$,
\begin{equation*}
 \lim_{k\rightarrow \infty}p_{\vec{k},\nu,\mathcal{I}}(\nu)=1.
\end{equation*}
Hence, there exist some constant $k_0>0$ such that for all $k>k_0$, we have
$$p_{\vec{k},\nu,\mathcal{I}}(\nu)\neq 0,$$
equivalently, $$\mathcal{P}_{\vec{k},\nu,\mathcal{I}}\nequiv 0,$$
which completes the proof of the theorem. 
\end{proof}
\begin{rmk}\label{rmkfinite}
More generally, the above theorem is equivalent to the following statement. \linebreak
\em{For any finite set $A \subset \mathcal{O}_F^*$ of indices and for any finite set $B$ of levels, there exist a positive constant $k_0$ such that for all $k>k_0$, we have
\begin{equation*}
\mathcal{P}_{\vec{k},\nu,\mathcal{I}}\nequiv 0,
\end{equation*}
for all $\nu\in A$ and $\mathcal{I}\in B$.}
\end{rmk}

In a particular case of the \thmref{nonvanishingw} when $F=\mathbb{Q}$, we get the following non-vanishing result of Poincar\'e series. 
\begin{cor}\label{nonvanishing corollary}
For any positive integer $m$, there exist a positive constant $k_0$ such that for all $k>k_0$,  we have
 \begin{equation*}
\mathcal{P}_{k,m,q}\nequiv 0,
\end{equation*}
where $\mathcal{P}_{k,m,q}$ be the $m$-${\rm{th}}$ Poincar\'e series of weight $k$ and of level $q$.
\end{cor}

Note that \corref{nonvanishing corollary} generalizes an earlier result of Gaigalas \cite{gai} and gives a completely different proof. Moreover, our result is stronger than that.

In the following theorem, we show the non-vanishing of Hilbert Poincar\'e series with respect to level, when the weight is fixed.
\begin{thm}\label{nonvanishingl}
Let $\nu\in\mathcal{O}_F^*$. Then for a fixed $2<k\in \mathbb{N}^n$ and $\nu$, there exist a positive constant $n_0$ such that
\begin{equation*}
\mathcal{P}_{k,\nu,\mathcal{I}}\nequiv 0,
\end{equation*}
for all integral ideal $\mathcal{I}$ with $N(\mathcal{I})\geq n_0$. 
\end{thm}
\begin{proof}
A direct application of \thmref{level} and similar arguments as in the proof of \thmref{nonvanishingw}, give the required result.
\end{proof}
\begin{rmk}
Similar to the \rmkref{rmkfinite}, we point out that the above theorem holds good for finitely many weights as well as finitely many indices.
\end{rmk}

Again if the field $F$ is $\mathbb{Q}$ then using \thmref{nonvanishingl}, we have the following non-vanishing result for Poincar\'e series $\mathcal{P}_{k,m,q}$.
\begin{cor}\label{nonvanishing corollary1}
For any positive integer $m$ and for a fixed weight $k\in \mathbb{N}$, there exist a positive constant $q_0$ such that for all $q>q_0$, we have
 \begin{equation*}
\mathcal{P}_{k,m,q}\nequiv 0.
\end{equation*}
\end{cor}

\section{concluding remark}\label{concluding}
It would be interesting to get the analogous result of Rankin, stated in the introduction, for Hilbert Poincar\'e series.
Although our results are weaker than the expected analogous Rankin's result, but the beauty of our proof is that it uses only elementary techniques.
We neither used the explicitly expression for the Fourier coefficients of Hilbert Poincar\'e series nor any estimates for corresponding Bessel function or the generalized Kloosterman sums. 

Note that the Fourier coefficients of Siegel Poincar\'e series also satisfy orthogonality relations, proved by Kowalski et. al. \cite{saha}. From this, one can obtained the similar results like \thmref{nonvanishingw} and \thmref{nonvanishingl} for Siegel Poincar\'e series. We also expect that by following the method of this paper, one can prove the analogous non-vanishing results for Poincar\'e series of half-integral weight modular forms.

\begin{acknowledgements} 
The author is grateful to her adviser Prof. Brundaban Sahu for his continuous support and useful discussions. She also express her gratitude to Prof. Abhishek Saha for  valuable comments on this paper. The author is greatly indebted to the anonymous referee
for his/her very thorough reading of the manuscript,
also resulting in improved readability of the present version.
\end{acknowledgements}

\end{document}